\documentclass[leqno,12pt]{amsart}

\usepackage{graphicx}
\usepackage{psfrag}
\usepackage{latexsym}
\usepackage{amssymb}
\usepackage{amsfonts}
\usepackage{amsmath}
\usepackage{amsthm}
\usepackage{color}
\usepackage[margin=1in]{geometry}
\usepackage{bm}

\newtheorem{theorem}{Theorem}[section]
\newtheorem{corollary}[theorem]{Corollary}

\newtheorem{proposition}[theorem]{Proposition}

\theoremstyle{definition}
\newtheorem{definition}[theorem]{Definition}

\theoremstyle{remark}
\newtheorem{remark}[theorem]{Remark}

\numberwithin{equation}{section}

\newcommand{\ep}{\epsilon}
\newcommand{\del}{\delta}

\newcommand{\reals}{\mathbb{R}}

\newcommand{\nats}{\mathbb{N}}

\newcommand{\inv}{^{-1}}

\newcommand{\ovl}{\overline}

\newcommand{\Hdim}{\mathcal{H}}

\newcommand{\subeq}{\subseteq}

\newcommand{\bslash}{\backslash}

\def\loc{\operatorname{loc}}
\def\diam{\operatorname{diam}}

\def\dist{\operatorname{dist}}

\def\W{\operatorname{W}}

\def\XXint#1#2#3{{\setbox0=\hbox{$#1{#2#3}{\int}$}
\vcenter{\hbox{$#2#3$}}\kern-.5\wd0}}

\author[Z.M. Balogh]{Zolt\'an M. Balogh}
\author[J.T. Tyson]{Jeremy T. Tyson}
\author[K. Wildrick]{Kevin Wildrick}
\address{Z. M. Balogh: Mathematisches Institut, Universit\"at Bern, Sidlerstrasse 5, 3012 Bern, Switzerland ({\tt balogh.zoltan@math.unibe.ch})}

\address{J. T. Tyson: Department of Mathematics, University of Illinois at Urbana-Champaign, 1409 W Green Street, Urbana, IL 61801, USA ({\tt tyson@math.uiuc.edu})}

\address{K. Wildrick: Department of Mathematical Sciences, Wilson Hall, Montana State University, Bozeman, MT 59717, USA ({\tt kevin.wildrick@montana.edu})}


\thanks{The first and third authors were supported by the Swiss National Science Foundation, European Research Council Project CG-DICE. The second author was supported by NSF grant DMS-1201875 and Simons Foundation Collaborative Grant \#353627.}

\begin{document}

\title[]{Quasiconformal mappings that highly distort dimensions of many parallel lines}

\date{\today}
\begin{abstract}
We construct a quasiconformal mapping of $\reals^n$, $n \geq 2$, that simultaneously distorts the Hausdorff dimension of a nearly maximal collection of parallel lines by a given amount. This answers a question of Balogh, Monti, and Tyson.
\end{abstract}

\maketitle
\begin{center}
\it Dedicated to Jussi V\"ais\"al\"a on the occasion of his 80\textsuperscript{th} birthday
\end{center}

\section{Introduction}

Despite their importance in a wide variety of mathematical settings, the family of quasiconformal mappings of $\reals^n$, $n > 2$, remains somewhat mysterious. Excellent introductions to the theory of quasiconformal mappings can be found in the monographs \cite{Resh} and \cite{VaisBook}. A core philosophical question is `how many such mappings are there?'. The richness of the class of quasiconformal mappings of $\reals^n$ is demonstrated in part by the existence of mappings which simultaneously and uniformly increase the Hausdorff dimension (denoted throughout this paper simply by $\dim$) of many leaves of a foliation of $\reals^n$. Such behavior, which cannot occur for smooth or even Lipschitz mappings, reflects the genuinely nonsmooth nature of quasiconformal mappings.

%One way to verify the richness of a given class of mappings defined on $\reals^n$, $n \geq 2$, is by producing a mapping in the class %that simultaneously and uniformly increases the Hausdorff dimension (denoted throughout simply by $\dim$) of many leaves of a foliation %of $\reals^n$; such a class of mappings must be much larger than, say, the class of Lipschitz mappings.

The results of \cite{BMT} provide bounds on the distortion of dimension of leaves of a foliation by a Sobolev mapping in terms of the desired dimension of the image of the leaf and the Sobolev exponent. Let us explain these bounds in a simplified setting. Let $n \geq 2$ be an integer, and let $L$ be any one-dimensional vector subspace of $\reals^n$. We consider the foliation $\{a+L\,:\,a \in L^\perp\}$ of $\reals^n$ by lines parallel to $L$. The absolute continuity along lines of a supercritical Sobolev mapping $f \in \W^{1,p}(\reals^n, \reals^N)$, $N \in \nats$, implies that
$$\dim(f(a+L)) = 1$$
for $\Hdim^{n-1}$-almost every $a \in L^\perp$. On the other hand, a folklore theorem (see, for instance, \cite{Kaufman}) states that
$$\Hdim^{p/(p-(n-1))}(f(a+L))=0$$
for \emph{any} (i.e., for $\Hdim^0$-almost every) $a \in L^\perp$. The following theorem from \cite{BMT} interpolates between these two results.

\begin{theorem}[Balogh--Monti--Tyson]\label{QC_lines}  
For $p>n \geq 2$, let $f \in \W^{1,p}(\reals^n;\reals^N)$, and let $\alpha \in \left(1,\frac{p}{p-(n-1)}\right]$. Then 
$\Hdim^{\alpha}(f(a+L))=0$ for $\Hdim^{\beta}$-almost every $a \in L^\perp$, where
$$\beta = (n-1)-p\left(1-\frac{1}{\alpha}\right).$$
\end{theorem}

According to Gehring's celebrated higher integrability theorem \cite{Gehring73}, each quasiconformal mapping of $\reals^n$ lies in $\W^{1,p}(\reals^n;\reals^n)$ for some $p>n$. Thus, Theorem \ref{QC_lines} has the following corollary.

\begin{corollary}[Balogh--Monti--Tyson]\label{qc_cor}
Let $f \colon \reals^n \to \reals^n$, $n \geq 2$ be a quasiconformal mapping. For each $\alpha \in (1,n)$ and for $\Hdim^\beta$-almost every $a \in L^\perp$, we have $\Hdim^{\alpha}(f(a+L))=0$, where $\beta=(n/\alpha)-1$.
\end{corollary}

Theorem \ref{QC_lines} is sharp in the following sense. Given $p>n$ and $\alpha$ and $\beta$ as in the statement, and for any integer $N > \alpha$, there exists a mapping $f \in \W^{1,p}(\reals^n, \reals^N)$ with the property that for $\Hdim^{\beta}$-almost every $a \in L^\perp$,
$$\dim(f(a+L)) = \alpha.$$
Such a mapping is constructed in \cite{BMT} by a random method (which is based on a construction of Kaufman). These mappings are unlikely to be injective, much less quasiconformal.

As evidenced by both the Riemann Mapping Theorem and the measurable Riemann Mapping Theorem, the class of quasiconformal mappings in $\reals^2$ is particularly rich, and so it is reasonable to expect that Corollary \ref{qc_cor} should also be sharp, at least when $n=2$. This expectation was confirmed by Bishop, Hakobyan, and Williams \cite{BHW}, who proved the following theorem.

\begin{theorem}[Bishop--Hakobyan--Williams]\label{BHW} 
Fix $\alpha \in [1,2)$ and a one-dimensional linear subspace $L \subset \reals^2$. For any $\beta < (2/\alpha)-1$, there is a set $E \subeq L^\perp$ and a quasiconformal mapping $f \colon \reals^2 \to \reals^2$ so that
\begin{equation}\label{haus} \dim(f(a+L)) \geq \alpha
\end{equation}
for each $a \in E$, where
\begin{equation}
 \dim(E) > \beta.
 \end{equation}
\end{theorem}
In fact, the authors of \cite{BHW} constructed a function $f$ as above with the property that 
$$\dim(f(F))=\alpha \dim(F)$$
for any $a \in E$ and for any Borel subset $F$ of $a+L$; this additional conclusion is substantially more interesting and more difficult to accomplish than the result which we described in Theorem \ref{BHW}. 

The construction in \cite{BHW} makes substantial use of conformal mappings and is therefore restricted to the planar case. When $n > 2$, the paucity of conformal mappings makes sharpness of Corollary \ref{qc_cor} much less clear. The paper \cite{BMT} contains a result analogous to Theorem \ref{BHW} for quasiconformal mappings in any dimension $n \geq 2$, but with Hausdorff dimension replaced by upper Minkowski dimension in \eqref{haus}. The images of all but countably many lines under that mapping are locally rectifiable and hence do not exhibit any Hausdorff dimension increase. A similar example with Hausdorff dimension in the target must necessarily proceed along different lines, and the authors of \cite{BMT} asked whether an optimal example exists.

In this work, we give a construction, different from all those mentioned above, that shows the sharpness of Theorem~\ref{QC_lines} for quasiconformal mappings in every dimension, taking into account the Sobolev exponent. Here is our main result.

\begin{theorem}\label{BMT sharp}  
Let $p>n \geq 2$, fix
$$\alpha \in \left[1,\frac{p}{p-(n-1)}\right),$$
and let $L$ be a one-dimensional linear subspace of $\reals^n$. For any
$$\beta <(n-1)- p\left(1-\frac{1}{\alpha}\right),$$
there is a set $E \subeq L^\perp$ with $\dim E > \beta$ and a quasiconformal mapping $\Phi$ in $\W^{1,p}(\reals^n;\reals^n)$ such that $\dim \Phi(L+a)> \alpha$ for each $a \in E$.
\end{theorem}

Our construction is a simple example of a `conformal elevator', an idea which appeared already in the work of Gehring and V\"ais\"al\"a \cite{GV} and has proven useful in dynamics \cite{Sullivan}, \cite{Pilgrim}, \cite{BonkMeyer}. Roughly speaking, we construct a single quasiconformal mapping between two multiply connected domains, then use iterated function systems of contracting similarities to recreate this mapping in the bounded complementary components of the domains. As we use rigid similarities rather than conformal mappings, it is more accurate to consider our construction as a `similarity elevator'. To verify quasiconformality somewhere deep in the construction, we ride the `similarity elevator' back to the original scale without accruing distortion.

The iterated function systems are chosen so that in the domain, the invariant set is subordinate to the foliation by copies of $L$, while in the target, combinatorial considerations ensure large dimension for the images of leaves of the foliation. A delicate honing of parameters yields a construction verifying Theorem \ref{BMT sharp}.

In the plane, the sharp relationship between the dilatation of a quasiconformal mapping and its Sobolev exponent was established by Astala \cite{Astala}. This relationship leads to sharp estimates of the Hausdorff dimension distortion of subsets in terms of the dilatation. However, Astala's dimension estimates are not sharp for lines; the sharp estimates were established by Smirnov \cite{Smirnov}. It would be very interesting to have a version of Theorem~\ref{BMT sharp} in which the role of the Sobolev exponent is assumed by the dilatation. As Iwaniec's conjecture \cite{Iwaniec} remains open, this seems tractable only when $n=2$.

The general theory of distortion of dimension of leaves of a foliation by Sobolev mappings can be extended to a large class of foliated metric spaces \cite{Metric}, and is of particular interest in the sub-Riemannian Heisenberg group \cite{Heis}. There, the sharpness of dimension distortion estimates analogous to Theorem \ref{QC_lines} remains unknown even for Sobolev mappings.

\

\paragraph{\bf Acknowledgements.} Research for this paper was initiated during the authors' attendance at the Eighth School in Analysis and Geometry in Metric Spaces held in Levico Terme, Italy in Summer 2014 and continued during a visit by the second and third authors to the University of Bern in Summer 2015. The hospitality of the Institute of Mathematics at the University of Bern is gratefully acknowledged.

\section{Quasiconformal mappings and iterated function systems}\label{IFS section}

\subsection{Pushing forward an iterated function system} In this section, we establish notation for iterated function systems in $\reals^n$ and describe how to build mappings of $\reals^n$ using such systems.

Let $\mathcal{I} = \{f_i \colon \reals^n \to \reals^n\}_{i=1}^N$ be a collection of contracting similarities of $\reals^n$ that satisfies the \emph{strong separation condition} \cite{Hochman}: there exists a bounded open set $Q$ such that for each $i=1,\hdots, N$, the set $f_i(\ovl{Q})$ is contained in $Q$, and the sets $\{f_i(\ovl{Q})\}_{i=1}^N$ are disjoint.
This condition implies that the invariant set of $\mathcal{I}$ is uniformly (and hence totally) disconnected.

Now, let $\mathcal{J} = \{g_i \colon \reals^n \to \reals^n\}_{i=1}^N$ and $S$ be another such collection and open set, and consider any continuous mapping
$$\phi \colon \reals^n \bslash \left(\bigcup_{i=1}^N f_i(Q)\right) \to \reals^n \bslash \left(\bigcup_{i=1}^N g_i(S)\right)$$
satisfying the \emph{compatibility condition}: there exists $\ep>0$ such that for each $y \in \reals^n \bslash Q$ with $\dist(y,Q) \leq \ep$ and each $i=1,\hdots, N$,
\begin{equation}\label{compatible} \phi(f_i(y)) = g_i(\phi(y).
\end{equation}
We will call such a mapping a \emph{generating mapping}.

To extend a generating mapping to a mapping on all of $\reals^n$ compatible with $\mathcal{I}$ and $\mathcal{J}$, we will employ the notation of symbolic dynamics. Let $\mathcal{S}$ be the collection of all finite and infinite sequences with entries in $\{1,\hdots,N\}$. The length of a sequence $\sigma \in \mathcal{S}$ is denoted by $|\sigma|$.
% Concatenation of two finite sequences $\sigma = (\sigma_1,\sigma_2,\hdots, \sigma_N)$ and $\tau=(\tau_1,\tau_2,\hdots,\tau_M)$ will be denoted as
%$$\sigma,\tau= (\sigma_1,\sigma_2,\hdots, \sigma_N, \tau_1,\tau_2,\hdots,\tau_M).$$
Given any sequence $\sigma \in \mathcal{S}$, we denote the initial sequence of $\sigma$ of length $k \in \nats$ by $\sigma|_k$. Finally, for $\sigma$ in $\mathcal{S}$ of finite length, we define a similarity $f_\sigma \colon \reals^n \to \reals^n$ by
$$f_{\sigma} = f_{\sigma_1}\circ f_{\sigma_2} \circ \cdots.$$

Let $\phi$ be a generating mapping. We define the \emph{generated mapping} $\Phi \colon \reals^n \to \reals^n$ inductively, as follows. First, we declare that  $$\Phi|_{\reals^n \bslash \left(\bigcup_{|\sigma|=1} f_\sigma(Q)\right)}=\phi.$$
Now, assume that for some integer $k \geq 1,$ the mapping $\Phi$ has been defined on
$$\reals^n \bslash \left(\bigcup_{|\sigma|=k} f_\sigma(Q)\right).$$
For each sequence $\sigma \in \Sigma$ of length $k$ and each $i=1,\hdots,N$, we define
$$\Phi|_{f_{\sigma}(\ovl{Q})\bslash \left(\bigcup_{i=1}^N f_{\sigma,i}(Q)\right)} = g_\sigma \circ \phi \circ f_\sigma\inv.$$
Finally, the invariant sets for the systems $\mathcal{I}$ and $\mathcal{J}$ are
$$K_{\mathcal{I}}= \bigcap_{k\in \nats }\bigcup_{|\sigma|=k}f_{\sigma}(Q) \ \ \text{and} \ \ K_{\mathcal{J}} =\bigcap_{k\in \nats }\bigcup_{|\sigma|=k}g_{\sigma}(S),$$
respectively. The strong separation condition guarantees each point of each invariant set can be uniquely identified with an infinite sequence in $\mathcal{S}$. As a result, $\Phi$ extends canonically to a bijection between $K_I$ and $K_J$, completing its definition.
\subsection{Quasiconformal generated mappings}
We employ the following metric definition of quasiconformality for mappings between subsets of $\reals^n$:
\begin{definition} Let $\Phi \colon \Omega \to \Omega'$ be a homeomorphism between subsets of $\reals^n$, and for all $x \in \Omega$ and $r > 0$, define
$$L_\Phi(x, r) :=\sup\{|\Phi(x)-\Phi(y)|: |x-y|\leq r, \ y\in \Omega \},$$
$$ l_\Phi(x,r):=\inf\{|\Phi(x)-\Phi(y)|: |x-y| \geq r, \ y\in \Omega \},$$
$$H_\Phi(x):=\limsup_{r \to 0} \frac{L_\Phi(x,r)}{l_\Phi(x,r)}.$$
The mapping $\Phi$ is \emph{$H$-quasiconformal}, $H \geq 1$, if $H_\Phi(x)\leq H$ for all $x \in \Omega$.
\end{definition}
A fundamental theorem of Gehring \cite{GehringDef} implies that the above definition coincides with other standard definitions of quasiconformal mappings if $\Omega$ is an open subset of $\reals^n$ and $\Phi$ is orientation preserving.

Note that if $f$ and $g$ are similarities of $\reals^n$, and $\phi \colon \Omega \to \Omega'$ is a homeomorphism of subsets of $\reals^n$, then for each $x \in \Omega$,
\begin{equation}\label{invariance} H_\phi(x) = H_{g \circ \phi \circ f\inv}(f(x)).
\end{equation}

For the remainder of this section, we consider iterated function systems $\mathcal{I}$ and $\mathcal{J}$ with the strong separation condition, a generating mapping $\phi$, and a generated mapping $\Phi$ as in the previous section, along with the relevant notation established there.
%We \emph{also assume that each similarity in $\mathcal{I}$ has the same contraction ratio $0<t<1$, and that each similarity in $\mathcal{J}$ has the same contraction ratio $0<\tau<1$}.

The goal of this section is to prove the following proposition. In this generality, the result does not seem to be present in the literature. However, the basic idea can be found in \cite[Theorem~5]{GV}.

\begin{proposition}\label{qc pasting} Suppose that the generating mapping $\phi$ is  quasiconformal homeomorphism. Then the generated mapping $\Phi \colon \reals^n \to \reals^n$ is a quasiconformal homeomorphism satisfying  $\Phi(K_\mathcal{I})=K_\mathcal{J}$.
\end{proposition}

The proof will show that the quasiconformality constant of $\Phi$ depends only on the quasiconformality constant of $\phi$, the scaling ratios associated to the iterated function systems $\mathcal{I}$ and $\mathcal{J}$, and a geometric quantity associated to the strong separation condition for $\mathcal{I}$ and $\mathcal{J}$.

Quasiconformality of $\Phi$ on the complement of the invariant set follows from the construction and the fact that the iterated function system is comprised of similarities.
%, we will ``ride the elevator" up to the top scale of the construction using \eqref{invariance}. Off of the invariant set %$K_{\mathcal{I}}$, we will then employ the quasiconformality of the generating map. 
On the invariant set, we use the fact that there are only finitely many open sets involved in the strong separation condition. This step differs from the approach taken in \cite{GV}, in which the open set of the strong separation condition is a cube. One can also verify quasiconformilty via removability results for quasiconformal mappings such as \cite[Theorem~4.2]{Def}, or in the planar setting \cite{Gotoh}. Such removability results are valid in wide generality, see \cite{BK} or \cite{BKR}.

%\begin{theorem}[Heinonen-Koskela] \label{annular porous} Let $E$ be a closed subset of $\reals^n$. If there is a constant $c>1$ and a sequence of scales $\{r_j\}_{j \in \nats}$ such that for each $x \in E$,
%$$E \cap \left(\ovl{B}(x,ar_j) \bslash \ovl{B}(x,r_j/a)\right) = \emptyset,$$
%then each $H$-quasiconformal embedding $f \colon \reals^n \bslash E \to \reals^n$ extends to an $H$-quasiconformal mapping $F \colon \reals^n \to \reals^n$.
%\end{theorem}

\begin{proof}[Proof of Proposition \ref{qc pasting}] We will show that for each $x \in \reals^n$,
\begin{equation}\label{PhiQC} H_{\Phi}(x)  \leq H.\end{equation}
This is immediate for points of
$$\reals^n \bslash \left(\bigcup_{i=1}^N \ovl{f_i(Q)}\right),$$
and by \eqref{invariance} it also holds at each point in the orbit of this set under the iterated function system of similarities $\mathcal{I}$. Moreover, \eqref{PhiQC} holds by definition at points of $\partial Q$, but this does not automatically imply \eqref{PhiQC} at points in the orbit of $\partial{Q}$, as the construction is ``glued together" at these points. To this end, suppose that $x \in \partial f_i(Q)$, where $1 \leq i \leq N$. If $0<r<\ep$,  then the compatibility condition \eqref{compatible} and the definition of the generated mapping $\Phi$ imply that
$$\Phi|_{B(x,r)} = g_i \circ \phi \circ f_i\inv|_{B(x,r)}.$$
Hence, \eqref{invariance} implies that
$$H_\Phi(x) = H_{\phi}(f_i\inv(x)) \leq H.$$
Thus, we only need to verify  \eqref{PhiQC} at points of $K_{\mathcal{I}}$.

%Suppose that $x \in \partial f_i(Q)$. Let $r>0$ be small enough that
%$$\dist(\partial Q, \partial f_i(Q)) > 2r,$$
%and define
%$$O(r)=\phi\left((\ovl{B}(x,r) \cap \left(\ovl{Q}\bslash f_i(Q)\right)\right) \ \text{and}\ I(r)=g_i\circ\phi\left(f_i\inv(B(x,r)) \cap \ovl{Q})\right).$$
%Then $\Phi(B(x,r)) = O(r) \cup I(r)$, and $$O(r) \cap I(r) = \phi(B(x,r)\cap \partial(f_i(Q))).$$
%Hence, we may estimate
%\begin{align*} L_{\Phi}(x,r) &=\sup\{d(f(x),y): y \in O(r) \cup I(r)\} \\
%& \leq \max\{L_{\phi}(x,r), L_{g_i \circ \phi \circ f_i\inv}(x,r)\}.
%\end{align*}
%Similarly,
%\begin{align*} l_{\Phi}(x,r) &= \inf\{d(f(x),y): y \in \ovl{S}\bslash \left(O(r) \cup I(r)\right)\} \\
%& \geq \min\{l_{\phi}(x,r), l_{g_i \circ \phi \circ f_i\inv}(x,r)\}.
%\end{align*}
%Since
%$$L_{g_i \circ \phi \circ f_i\inv}(x,r) =\tau L_{\phi}\left(f_i\inv(x),\frac{r}{t}\right), \ \text{and} \ l_{g_i \circ \phi \circ f_i\inv}(x,r) = \tau l_{\phi}\left(f_i\inv(x),\frac{r}{t}\right).$$
%This implies that $H_{\Phi}(x)\leq H$.

%We claim that the hypotheses of Theorem \ref{annular porous} are satisfied when $E=K_{\mathcal{I}}$.

 Finally, let $x \in K_{\mathcal{I}}$. Associated to $x$ is a unique infinite sequence $\sigma \in \mathcal{S}$, so that
$$\{x\} = \bigcap_{j \in \nats} f_{\sigma|_j}(\ovl{Q}),$$
where $\sigma|_j$ is the truncation of $\sigma$ to length $j \in \nats$.
For $i=1,\hdots, N$, let $t_i$ and $\tau_i$ be the scaling ratios of the similarities $f_i$ and $g_i$, respectively. Moreover, for any integer $j \geq 1$, denote the scaling ratios of $f_{\sigma|_j}$ and $g_{\sigma|_j}$ by $t_{\sigma|_{j}}$ and $\tau_{\sigma|_{j}}$.

Set
$$d = \frac{1}{2} \dist\left(\bigcup_{i=1}^N(f_i(\ovl{Q})), \reals^n\bslash Q\right).$$
Then for any integer $j \geq 1$,
$$t_{\sigma|_j} d  < \dist\left(f_{\sigma|_{j+1}}(\ovl{Q}), \reals^n\bslash f_{\sigma|_j}(Q)\right),$$
which implies that
\begin{equation}\label{incl} B(x,t_{\sigma|_{j}}d) \subeq f_{\sigma|_j}(Q) \subeq B(x, 2t_{\sigma|_{j}}\diam Q).
\end{equation}
By the definition of $\Phi$, it holds that
$$\{\Phi(x)\} = \bigcap_{j \in \nats} g_{\sigma|_j}(\ovl{S}).$$
Setting
$$\del = \frac{1}{2} \dist\left(\bigcup_{i=1}^N(g_i(\ovl{S})), \reals^n\bslash S\right),$$
an analogous argument shows that for any integer $j \geq 1$,
\begin{equation}\label{image incl/supcl} B(\Phi(x),\tau_{\sigma|_{j}}\del) \subeq g_{\sigma|_{j}}(S) \subeq B(x,2\tau_{\sigma|_{j}}\diam{S})).
\end{equation}

Choose an integer $\kappa \geq 1$ such that
$$\left(\max_{i=1,\hdots, N} t_i \right)^{\kappa-1} < \frac{d}{2\diam Q}.$$
Note that $\kappa$ depends only on the iterated function systems $\mathcal{I}$ and $\mathcal{J}$, and not on $x$.

Then, for any integer $j \geq 1$,
$$2t_{\sigma|_{j+\kappa}}\diam Q < t_{\sigma|_{j+1}} d.$$
Thus, if $t_{\sigma|_{j+1}}d \leq r < t_{\sigma|_{j}}d,$
then \eqref{incl}, \eqref{image incl/supcl}, and the definition of $\Phi$ imply that
$$B(\Phi(x), \tau^{j+1+\kappa} \del) \subeq g_{\sigma|_{j+1+\kappa}}(S) \subeq \Phi(B(x,r)) \subeq g_{\sigma|_{j}}(S) \subeq B(\Phi(x),2\tau^j \diam S).$$
This yields
$$\frac{L_{\Phi}(x,r)}{l_\Phi(x,r)} \leq \frac{2\diam S}{\tau^{1+\kappa}\del}.$$
As the integer $j \geq 1$ was arbitrary, this shows that
$$H_{\Phi}(x) \leq \frac{2\diam S}{\tau^{1+\kappa}\del}$$
as well, completing the proof.
\end{proof}
\begin{remark}\label{no bounds} If $\partial{Q}$ is assumed to have $\sigma$-finite $(n-1)$-dimensional Hausdorff measure, then the compatibility condition \eqref{compatible} can be weakened to: for each $y \in \partial{Q}$ and each $i=1,\hdots, N$,
$$\phi(f_i(y))=g_i(\phi(y)).$$
This is due to classical removability results, which can be found in \cite{GehringDef}, \cite{GehringRings}.
\end{remark}

\subsection{Almost full iterated function systems with the strong separation condition}

It is not possible to find 15 disjoint closed squares of side-length $1/4$ inside an open square of side-length $1$, even though the total area of the smaller squares is less than the area of the larger square. In this section, we will show that by beginning with a product of intervals that is not a cube, we may pack as many scaled copies as is allowed by volume considerations. This will be a crucial part of our construction.

\begin{proposition}\label{rectangles} Let $n \geq 1$ be an integer, and let $0<r<1$. For each integer $1 \leq M<r^{-n}$, there is an iterated function system of orientation preserving contracting similarities $\mathcal{I}=\{f_i \colon \reals^n \to \reals^n\}_{i=1}^M$, each with scaling ratio $r$, that satisfies the strong separation condition on a product of $n$ bounded open intervals.
\end{proposition}

\begin{proof} The statement is trivial if $n=1$, and so we assume that $n \geq 2$. Choose positive numbers $h_1, \hdots, h_n=1$ such that
$$\frac{1}{r^n}= \frac{h_n}{r^{n}} > \frac{h_{n-1}}{r^{n-1}} > \hdots > \frac{h_{1}}{r} > M.$$
Let $$Q=\left(\prod_{i=1}^{n-1}(0,h_i)\right) \times (0,1) \subeq \reals^n.$$
Denote by $F \colon \reals^n \to \reals^n$ an orientation preserving similarity of $\reals^n$ that has scaling ratio $r$ and permutes the coordinate axes
$$L_i:=\{(x_1,\hdots,x_n) \in \reals^n : x_j = 0 \ \text{if}\ j\neq i\}, \ i=1,\hdots,n$$ as follows
$$L_n \to L_{1} \to \hdots \to L_{n-1} \to L_n.$$
Since $Mr<h_{1}$, there are $M$ disjoint closed intervals of length $r$ inside the open interval $(0,h_{1})$. Similarly, for each $i=1, \hdots n-1$, since $rh_{i}<h_{i+1}$, there is a disjoint closed interval of length $rh_{i}$ insider the open interval $(0,h_{i+1}).$ This implies that we may find $M$ vectors $v_1,\hdots v_M$ such that the similarities
$$\{f_i = F+ v_i\}_{i=1}^M$$
satisfy the desired strong separation condition on $Q$.
\end{proof}

\section{Quasiconformal mappings of maximal frequency of dimension distortion}

We now apply Theorem \ref{qc pasting} to prove Theorem \ref{BMT sharp}. We assume, without loss of generality, that $L$ is the one-dimensional subspace of $\reals^n$ generated by the last coordinate direction. More precisely, define $\pi \colon \reals^n \to \reals^{n-1}$ by
$$\pi(x_1,\hdots,x_n) \mapsto (x_1,\hdots,x_{n-1}).$$
Then $L$ is the kernel of $\pi$, and for each $a \in \reals^{n-1}$,
$$(a,0)+L = \pi\inv(a).$$
We denote the complementary projection by $\pi^\perp \colon \reals^n \to \reals$, where
$$\pi^\perp(x_1,\hdots,x_n) = x_n.$$

\begin{proof} As in the statement of Theorem~\ref{BMT sharp}, we fix $p>n\geq 2$ and
$\alpha \in \left[1, \frac{p}{p-(n-1)}\right).$
Let
$$0< \beta < \hat{\beta}:=(n-1)- p\left(1-\frac{1}{\alpha}\right).$$
Then
$$ (n-1)-\beta  >0, \ \text{and}\  \frac{n}{\alpha}-\beta - 1 > 0.$$
We will choose a parameter $d$ so that several inequalities are satisfied. Each inequality will hold when $d$ is sufficiently small; rather than choose the smallest requirement, we explicitly list the requirements separately for the reader's convenience. Specifically, choose $d>0$ so that
\begin{equation}\label{min} d < \min \begin{cases}
											\left(\frac{1}{2}\right)^{1/\beta} & (1),\\
											\left(2^\beta -1 \right)^{1/\beta} & (2),\\
											 2^{-\beta/(n-1-\beta)} & (3),\\
											2^{-(1+\alpha)}& (4),\\
											 1-2^{-\alpha} & (5), \\
											 \left(2^{-\beta}3^{-n}\right)^{1/(\frac{n}{\alpha}-\beta - 1)} & (6),\\
											\left(2^{-\beta}3^{-p}\right)^{1/(\hat{\beta} - \beta)} & (7). \\
\end{cases}\end{equation}
Employing terms (1)-(3) in \eqref{min}, we see that there is an integer $M$ so that the following inequalities are satisfied:
\begin{equation}\label{M} 2 < \left(\frac{1}{d}\right)^\beta < M <  \left(\frac{2}{d}\right)^\beta < \left(\frac{1}{d}\right)^{n-1}.\end{equation}
Terms (4)-(5) in \eqref{min} allow us to find a number $t>0$ so that
\begin{equation}\label{t} 2d^{1/\alpha} < t < \min\left\{2^{-1/\alpha}, \left( 2^\alpha -1 \right)^{1/\alpha}, \ 3d^{1/\alpha} \right\}.\end{equation}
Hence there is an integer $M'$ so that the following inequalities are satisfied:
\begin{equation}\label{M'} 2 < \left(\frac{1}{t}\right)^\alpha \leq M' <  \left(\frac{2}{t}\right)^\alpha < \left(\frac{1}{d}\right).\end{equation}
Terms (6) and (7) in \eqref{min} will be employed later in the construction.

Since $M < d^{-(n-1)}$, Proposition \ref{rectangles} yields an iterated function system
$$\mathcal{K}=\{h_1,\hdots,h_{M} \colon \reals^{n-1} \to \reals^{n-1}\}$$
of orientation preserving contracting similarities with contraction ratio $d$ that satisfies the strong separation  condition on a product $Q_{n-1}$ of $(n-1)$ bounded open intervals in $\reals^{n-1}$. Similarly, since $M' < d\inv$, we may find an iterated function system $$\mathcal{K}'=\{h'_1,\hdots,h'_{M'}\colon \reals \to \reals\}$$
of orientation preserving contracting similarities with ratio $d$ that satisfy the strong separation condition on $(0,1)$.  Denote by $K_{\mathcal{K}}$ and $K_{\mathcal{K}'}$ the invariant sets of $\mathcal{K}$ and $\mathcal{K}'$. Since the strong separation condition implies the open set condition, the Moran--Hutchinson theorem \cite{Moran}, \cite{Hutchinson} implies that similarity dimension and the Hausdorff dimension of $K_{\mathcal{K}}$ agree; the same is true of $K_{\mathcal{K}'}$. Hence, \eqref{M}, and \eqref{t} imply that
$$\dim K_{\mathcal{K}} = \frac{\log M}{\log d\inv} > \beta, \ \text{and} \ \dim  K_{\mathcal{K}'} = \frac{\log M'}{\log d\inv} < 1.$$

We define the \emph{product} iterated function system
$$\mathcal{K}\times \mathcal{K}' = \{h_{i,j} \colon \reals^n \to \reals^n: 1\leq i \leq M,\ 1 \leq j \leq M'\}$$ by setting
$$h_{i,j}(x_1,\hdots,x_n) = (h_i(x_1,\hdots,x_{n-1}), h'_j(x_n)).$$
Note that $\mathcal{K} \times \mathcal{K}'$ consists of orientation preserving contracting similarities of ratio $d$ and satisfies the strong separation condition on $Q:=Q_{n-1} \times (0,1)$.

We now claim that there is an iterated function system of orientation preserving contracting similarities
$$\mathcal{J} = \{g_{i,j} \colon \reals^n \to \reals^n: 1\leq i \leq M, \ 1\leq j \leq M'\}$$
that again satisfies the strong separation condition on product $S$ of $n$ bounded open intervals and is combinatorially equivalent to $\mathcal{K} \times \mathcal{K'}$, i.e., it contains the same number of mappings at the first iteration, but so that each mapping in $\mathcal{J}$ has the larger contraction ratio $t$ defined above. By Proposition \ref{rectangles}, this can be accomplished if
\begin{equation}\label{target restriction} MM'< t^{-n}.\end{equation}
Towards this end, note that \eqref{M}, \eqref{M'}, and \eqref{t}, imply that
$$MM't^n < \left(\frac{2}{d}\right)^\beta\left(\frac{1}{d}\right) \left(3d^{1/\alpha}\right)^n \leq  2^{\beta}3^nd^{\frac{n}{\alpha}-\beta-1}.$$
Hence, term (6) of \eqref{min} verifies \eqref{target restriction}.

Moreover, since $Q$ and $S$ are both the products of $n$ bounded open intervals, the strong separation condition allows us to produce a piecewise-linear and orientation preserving quasiconformal generating mapping $$\phi \colon \reals^n \bslash \left( \bigcup_{i,j} h_{i,j}(Q) \right) \to \reals^n \bslash \left( \bigcup_{i,j} g_{i,j}(S) \right).$$
Thus, we may apply Theorem \ref{qc pasting} to produce a quasiconformal mapping $\Phi \colon \reals^n \to \reals^n$ that canonically maps $K_{\mathcal{K}\times \mathcal{K}'}$ to $K_{\mathcal{J}}$.

The key point of this construction is that the invariant set $K_{\mathcal{K}\times \mathcal{K}'}$ is the product $K_\mathcal{K} \times K_{\mathcal{K}'}$. Hence, for each point $a \in K_\mathcal{K}$, the intersection of $\pi\inv(a)$ with $K_{\mathcal{K} \times \mathcal{K}'}$ is a copy of $K_{\mathcal{K}'}$. This set is then mapped by the quasiconformal mapping $\Phi$ onto a combinatorially equivalent Cantor set. The dimension of this Cantor set is specified by our choice the contraction ratio $t$ of the system $\mathcal{J}$.

For ease of notation, we now denote by $\mathcal{S}$ the collection of all finite and infinite sequences with entries in $\{1,\hdots M\}$ and by $\mathcal{S}'$ the collection of all finite and infinite sequences with entries in $\{1,\hdots,M'\}$. Given sequences $\sigma \in \mathcal{S}$ and $\tau \in \mathcal{S}'$, each of length $k \in \nats$, we write
$$
h_{(\sigma,\tau)} = h_{\sigma_k,\tau_k} \circ \hdots \circ h_{\sigma_1,\tau_1} \qquad \text{and} \qquad
g_{(\sigma,\tau)} = g_{\sigma_k,\tau_k} \circ \hdots \circ g_{\sigma_1,\tau_1}.
$$
Let $a \in K_\mathcal{K}$. By construction, there is an infinite sequence $\sigma \in \mathcal{S}$ such that
$$\{a\} = \bigcap_{k \in \nats} h_{\sigma|_k}(\pi(\ovl{Q}), $$
and hence
$$\pi\inv(a) \cap K_{\mathcal{K} \times \mathcal{K}'} =  \bigcap_{k \in \nats} \left( \bigcup_{\tau \in \mathcal{S}', |\tau|=k}
h_{(\sigma|_k,\tau)}(\ovl{Q})\right).$$
By construction, for each $k \in \nats$ and each $\tau \in \mathcal{S}'$ of length $k$,
$$\Phi \circ h_{(\sigma|_k,\tau)}(\ovl{Q}) = g_{(\sigma|_k,\tau)}(\ovl{S}).$$
Hence,
$$\Phi(\pi\inv(a) \cap K_{\mathcal{K} \times \mathcal{K}'}) = \bigcap_{k \in \nats} \left( \bigcup_{\tau \in \mathcal{S}', |\tau|=k} g_{(\sigma|_k,\tau)}(\ovl{S})\right).$$
A standard argument (see, e.g., \cite[Section~4.12]{Mattila}) and \eqref{t} now imply that dimension of this Cantor-type set satisfies
$$\dim \Phi(\pi\inv(a)\cap K_{\mathcal{K} \times \mathcal{K}'}) = \frac{\log M'}{\log t\inv} > \alpha.$$
Hence, for each point $a \in K_\mathcal{K}$, the image of the fiber $\pi\inv(a)$ under the map $\Phi$ has dimension greater than $\alpha$.

Since $\Phi$ is a piece-wise linear homeomorphism on $\ovl{Q} \bslash \bigcup_{i,j} h_{i,j}(Q)$,
$$\int_{\ovl{Q} \bslash \bigcup_{i,j} h_{i,j}(Q)} |D\Phi|^p \ d\Hdim^n =: C< \infty.$$
We may now conclude from the self-similar nature of $\Phi$ that
$$\int_{\ovl{Q}} |D\Phi|^p \ d\Hdim^n \leq C\sum_{l=0}^\infty \left(MM'\left(\frac{t}{d}\right)^{p}d^{n}\right)^l$$
Again applying \eqref{M}, \eqref{t}, and \eqref{M'}, we estimate
$$MM'\left(\frac{t}{d}\right)^{p}d^{n} < \left(\frac{2}{d}\right)^\beta \left(\frac{1}{d}\right) (3d^{\frac{1}{\alpha}})^{p} d^{n-p} = 2^\beta 3^p d^{\hat{\beta} - \beta}.$$
Term (7) of \eqref{min} now implies that $MM'(t/d)^{p}d^{n}<1$ and so
$$\int_{\ovl{Q}} |D\Phi|^p \ d\Hdim^n <\infty.$$
Since $\Phi$ can be chosen to be linear off $\ovl{Q}$, this implies that $\Phi \in \W^{1,p}_{\loc}(\reals^n;\reals^n).$
\end{proof}
\bibliographystyle{plain}
\bibliography{QuasiDistortion}

\end{document}